\documentclass{amsart}

\usepackage{amsmath,amssymb,amsthm}
\usepackage{graphicx}
\usepackage{color}
\usepackage{hyperref}

\DeclareMathOperator{\tb}{tb}
\DeclareMathOperator{\rot}{rot}
\DeclareMathOperator{\lk}{lk}

\DeclareMathOperator{\writhe}{writhe}

\newtheorem{theorem}{Theorem} [section]
\newtheorem{thm}[theorem]{Theorem}

\newtheorem{corollary}[theorem]{Corollary}
\newtheorem{proposition}[theorem]{Proposition}
\newtheorem{lemma}[theorem]{Lemma}

\theoremstyle{definition}
\newtheorem{defn}[theorem]{Definition} 
\newtheorem{remark}[theorem]{Remark} 
\newtheorem{example}[theorem]{Example}

\begingroup\expandafter\expandafter\expandafter\endgroup
\expandafter\ifx\csname pdfsuppresswarningpagegroup\endcsname\relax
\else
\fi

\begin{document}


\title[Rotation Numbers and the Euler Class in Open Books]{Rotation Numbers and the Euler Class in Open Books}

\author{Sebastian Durst}
\address{Fraunhofer-Institut f\"ur Hochfrequenzphysik und Radartechnik FHR,
Fraunhoferstr. 20, 53343 Wachtberg, Germany}
\email{sdurst@math.uni-koeln.de}

\author{Marc Kegel}
\address{Institut f\"ur Mathematik, Humboldt-Universit\"at zu Berlin, Unter den Linden 6, 10099 Berlin, Germany}
\email{kegemarc@math.hu-berlin.de}

\author{Joan E. Licata}
\address{Mathematical Sciences Institute, The Australian National University, ACT 0200}
\email{joan.licata@anu.edu.au}


\begin{abstract}
This paper introduces techniques for computing a variety of numerical invariants associated to a Legendrian knot in a contact manifold presented by an open book with a Morse structure.  Such a Legendrian knot admits a front projection to the boundary of a regular neighborhood of the binding.  From this front projection, we compute the rotation number for any null-homologous Legendrian knot as a count of oriented cusps and linking or intersection numbers; in the case that the manifold has non-trivial second homology, we can recover the rotation number with respect to a Seifert surface in any homology class.  We also provide explicit formulas for computing the necessary intersection numbers from the front projection, and we compute the Euler class of the contact structure supported by the open book.  Finally, we introduce a notion of Lagrangian projection and compute the classical invariants of a null-homologous Legendrian knot from its projection to a fixed page.

\end{abstract}

\date{\today} 



\maketitle

\section{Introduction}

The rotation and Thurston-Bennequin numbers  are defined for oriented null-homologous Legendrian knots in arbitrary contact $3$-manifolds.  In the case of a Legendrian knot $\Lambda$ in the standard contact $3$-space, which we denote by $(\mathbb{R}^3, \xi_\text{std})$, these can be computed via simple combinatorial formulas from either the front or Lagrangian projection of $\Lambda$.   Analogous techniques have been developed for Legendrians in special classes of $3$-manifolds that admit natural notions of projections, and or in the case of Legendrians embedded in a single page of an open book \cite{SO18, EtOz08, DuKeKl16, DuKe18}.  This paper establishes techniques for computing the classical invariants of a generic Legendrian knot in an arbitrary contact $3$-manifold.  A second key result of the paper is a technique for computing the Euler class of the contact structure itself; see Theorem~\ref{thm:eu}.  Although there are other approaches to computing invariants of a contact structure viewed as a plane field,  the technique introduced here  is markedly simpler than the existing methods we are aware of (for example, \cite{EtOz08} and \cite{DG2004}).

The notion of a \textit{Morse structure} on an open book presenting a contact manifold was introduced in  \cite{GaLi15} (and is reviewed in Section~\ref{sec:back}). A Morse structure determines a $2$-complex, called the \textit{skeleton},  and a vector field $V$.  Given  a Legendrian knot $\Lambda$ transverse to the skeleton and disjoint from the binding, the \textit{front projection} $\mathcal{F}(\Lambda)$ is the image of $\Lambda$ under the flow by $\pm V$ to a collection of tori embedded in $M$. The resulting curves share many properties with the front projection of a Legendrian knot in $(\mathbb{R}^3, \xi_\text{std})$.  For example, it was shown in \cite{GaLi15} that the Thurston-Bennequin number of a null-homologous Legendrian knot can be computed from its front projection.  The first result in this paper establishes  combinatorial techniques for computing the rotation number from $\mathcal{F}(\Lambda)$.

Throughout the paper, let $\Lambda\subset (M,\xi)$ be an oriented null-homologous Legendrian knot in a contact $3$-manifold presented via an open book $(B, \pi)$ with a Morse structure. The vector field $V$ is parallel to each page $\Sigma_t=\pi^{-1}(t)$ of the open book, and we let $L_i$ denote the link in $M$ consisting of the critical points of $V|_{\Sigma_t}$ with index $i$, oriented by increasing $t$ value. Define $\mathcal{L}$ to be the link $L_0\cup -L_1$. See Section~\ref{sec:back} for essential definitions related to Morse structures and front projections.

\begin{thm}\label{thm:front} Suppose that $M$ is an integer homology $3$-sphere.  If $D$ and $U$ denote the number of cusps of $\mathcal{F}(\Lambda)$ oriented down and up, respectively, then
\[ \text{rot}(\Lambda)=\frac{1}{2}[D-U]+\text{lk}(B\cup \mathcal{L},\Lambda).\] 
\end{thm}

 Theorem~\ref{thm:front} generalizes to arbitrary $3$-manifolds.  In this setting, the rotation number of a Legendrian knot depends on a choice of the homology class of a Seifert surface;  we first explain how to identify a relative second homology class  and then compute the associated rotation number. Throughout the paper, all homology will be taken with integer coefficients.

In arbitrary $3$-manifolds, it will be convenient to replace a Legendrian knot $\Lambda$ with a link $\Lambda \cup X$ that is null-homologous in the cylinder $\Sigma \times I$.   It follows from Remark 6.3 in \cite{GaLi15} (and  the proof of Lemma~\ref{lem:intB} below)  that  $\Lambda$ may be isotoped through Legendrians until it is disjoint from the page $\Sigma_0$ of the open book.  The standing assumption that $\Lambda$ is null-homologous allows us to choose a Seifert surface $H$ for $\Lambda$. Let $X:=H\cap \partial \big(M\setminus N(\Sigma_0)\big)$.  By construction,  $\Lambda \cup X$ is null-homologous in the  cylinder and any Seifert surface for $\Lambda\cup X$ may be glued along $X$ to produce a Seifert surface for $\Lambda$ in the class of $H$.    (We make no assumption that the components of $X$ are themselves Legendrian, but if desired, one may cut the original surface $H$ along a Legendrian approximation of $H\cap \Sigma_0$ and then isotope this into the interior of $M \setminus \Sigma_0$.)

\begin{lemma}\label{lem:Xdet} The link $X$ uniquely determines $[H]$ as an element of $H_2(M, \Lambda)$.
\end{lemma}

In the next theorem and throughout the paper, we write $A\bullet B$ to denote the intersection product of $A$ and $B$.  Since $A$ and $B$ will  always be oriented submanifolds which may be isotoped to intersect transversely,  $A\bullet B$ is simply the algebraic intersection number.

\begin{thm}\label{prop:front} Suppose that $\Lambda\subset M\setminus \Sigma_0$ and let $X$ be an auxiliary link with the property that $[\Lambda\cup X]=0\in H_1(\Sigma\times I)$. Let $D$ and $U$ denote the number of cusps of $\mathcal{F}(\Lambda)$ oriented down and up, respectively.  Let $H_X$ be any Seifert surface for $\Lambda$ in the class in $H_2(M, \Lambda)$ determined by $X$.  Then the rotation number of $\Lambda$ with respect to the second homology class determined by $X$ is computed by the following:
 \[ \text{rot}_X(\Lambda)=\frac{1}{2}[D-U] + \lk(B,\Lambda)+ \mathcal{L} \bullet H_X\] 
\end{thm}

	If $\mathcal{L}$ is null-homologous, $\mathcal{L}\bullet H_X=\lk(\mathcal{L},\Lambda)$ and we recover Theorem~\ref{thm:front}. In particular, the rotation number is independent of the choice of Seifert surface if $\mathcal{L}$ is null-homologous.  
	
	Viewed simply as a two-plane field on a $3$-manifold, a contact structure $\xi$ has an Euler class $e(\xi)$.

\begin{thm}\label{thm:eu} The link $\mathcal{L}$ is Poincar\'e dual to  $e(\xi)$.
\end{thm}

The proof of Theorem~\ref{thm:eu} is essentially a corollary of Theorem~\ref{prop:front}.  In Section~\ref{sec:eucomp}, we show how to compute $[\mathcal{L}]\in H_1(M)$ from the Morse diagram for $M$. Although we do not explore this further in the present work, we note that the computational techniques described in Section~\ref{sec:eucomp} may also be applied to spinnable foliations as in \cite{KMMM09}.

Turning Theorems~\ref{thm:front} and \ref{prop:front} into effective tools requires computing intersection numbers between Legendrian curves and Seifert surfaces representing prescribed homology classes.  The combinatorial formulas which accomplish this are easier to state once further notation has been introduced, so we offer the next theorem as a place holder for the more precise version stated and proven in Section~\ref{sec:linking}.

\begin{thm}\label{thm:linking1} Suppose $\Lambda_1$ is contained and null-homologous in $M\setminus \Sigma_0$.  Then the intersection number between $\Lambda_2$ and a Seifert surface for $\Lambda_1$ may be computed combinatorially from the front projections of the $\Lambda_i$.
\end{thm}

Note that for an arbitrary $\Lambda$ which is null-homologous in $M$ and a specified homology class of Seifert surface, the addition of the link $X$ described before Lemma~\ref{lem:Xdet} ensures that $\Lambda_1:=\Lambda \cup X$ satisfies the hypotheses of Theorem~\ref{thm:linking1}.

In the final section of the paper, we introduce a notion of Lagrangian projection for Legendrians disjoint from a fixed page $\Sigma_0$.  The existence of a Morse structure requires a \textit{monodromy vector field} on $M\setminus B$.  If $\Lambda$ is disjoint from $\Sigma_0$, we may push $\Lambda$ along the monodromy vector field to $\Sigma_0$; we call the image  of this operation the Lagrangian projection $\Lambda_\Sigma$ and we show that both the Thurston-Bennequin and rotation numbers can be computed combinatorially from $\Lambda_\Sigma$.

\begin{proposition}\label{lem:classicalInvariantsInLagrangian} Suppose that $\Lambda$ is a Legendrian knot contained and null-ho\-mo\-lo\-gous in $\Sigma\times I$. Then we have the following:

\begin{enumerate}
\item  The Thurston-Bennequin number of $\Lambda$ is 
\begin{equation*}
\tb(\Lambda)=\writhe_\Sigma(\Lambda_\Sigma); \text{ and}
\end{equation*}

\item the rotation number is \[\rot(\Lambda)=\rot_{\Sigma}(\Lambda_\Sigma)= \rot_{V_0}(\Lambda_\Sigma) + \mathcal{L}\bullet H_\emptyset\] where $\rot_{\Sigma}(\Lambda_\Sigma)$ is the rotation computed using an arbitrary non-vanishing vector field on $\Sigma$, $\rot_{V_0}(\Lambda_\Sigma)$ is the rotation computed with respect to the vector field coming from the Morse structure, and $H_\emptyset$ is a Seifert surface disjoint from $\Sigma_0$.

\end{enumerate}
\end{proposition}

Section~\ref{sec:lagr}  explains how to compute $\mathcal{L}\bullet H_\emptyset$ directly from a Lagrangian projection and discusses the relationship between the formulas in the front and the Lagrangian projection.

\begin{remark}
Although the  results in this paper are generally stated for knots, they extend easily to the case of links.  
\end{remark}

\subsection*{Acknowledgments}

Part of this research was carried out during a visit of the third author to Cologne, and we gratefully acknowledge the support of the graduate school of the Mathematical Institute of the University of Cologne.  The second author is supported by the Berlin Mathematical School.

\section{Background and conventions}\label{sec:back}

Suppose that $(B,\pi)$ is an open book supporting a contact structure $\xi$ on $M$. A \textit{Morse structure} compatible with $(B,\pi, \xi)$ is a pair $(F,V)$ consisting of a certain smooth function $F$ and vector field $V$.  The function $F$ restricts to each page as a Morse function $f_t$ with a single index $0$ critical point, finitely many index $1$ critical points, and no index $2$ critical points.  The vector field $V$ is tangent to each page $\Sigma_t$, and the restriction $V_t$ is both gradient-like for $f_t$ and Liouville for a symplectic form on $\Sigma_t$ associated to $\xi$. The precise definition, which can be found in~\cite{GaLi15}, includes a number of technical conditions that  control the behavior at the binding and  how $(f_t, V_t)$ evolves with $t$.  For the present purposes, it's enough to note that the ascending and descending manifolds of $V$ intersect transversely in level sets, except on finitely many pages where handle slides occur.

We assume throughout that all Morse structures are compatible with contact structures supported by the given open book; note that this combines Definitions 3.1 and 4.1 in \cite{GaLi15}.

We denote an open book with a compatible Morse structure by $(B, \pi, F, V)$, and the \textit{skeleton} of $(B, \pi, F, V)$ is the union of flowlines of $V$ between index $0$ and index $1$ critical points of $f_t$.  The intersection of the skeleton with each page is a finite graph, while the union over all $t$ is a $2$-complex.  The \textit{co-skeleton} of $(B, \pi, F, V)$ is the union of flowlines of $V$ between index $1$ critical points of $f_t$ and $B$.  We remark that in this language, the link $\mathcal{L}$ is the singular set of the complex Skel $\cup$ Co-Skel,  appropriately oriented.

 The \textit{Morse diagram} associated to a Morse structure is a collection of embedded tori $\partial N(B)=\coprod T_i$, decorated with $\coprod T_i \cap \text{Co-Skel}$.  For generic $t$-values, the co-skeleton intersects the collection of tori in $2k$ disjoint points, identified in pairs associated to flowlines from the same index $1$ critical point.  As $t$ varies, these trace out paired curves on the Morse diagram.  These curves are continuous except at $t$-values corresponding to handle slides in $f_t$, where the sliding co-core breaks into two points which are coincident with another pair.  See Figure~\ref{fig:firstex}.   We call these coincident endpoints \textit{teleporting points}.

\begin{figure}[h!]
	\centering
	\includegraphics[width=\textwidth]{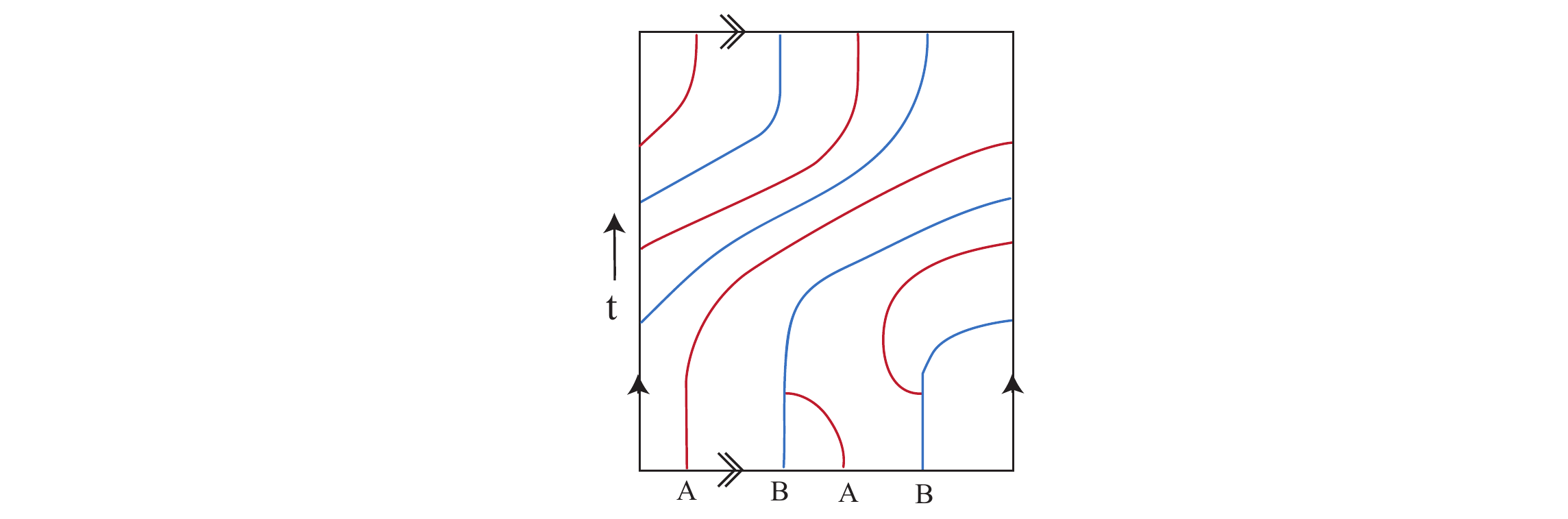}
	\caption{A Morse diagram for an open book whose page is a once punctured torus.  There is one handle slide followed by a boundary-parallel Dehn twist.}
	\label{fig:firstex}
\end{figure}

Propositions 3.3 and  4.3 in \cite{GaLi15} show that every open book has a Morse structure compatible with the supported contact structure, and Proposition 3.7 in~\cite{GaLi15} states that the original contact manifold can be recovered from the Morse diagram. Theorem 1.1 in~\cite{GaLi15} states that in any $(B, \pi, F, V)$, the complement of the skeleton is a collection of  solid tori. Even more explicitly, after removing both the binding and the skeleton, each connected component is contactomorphic to \begin{equation}\label{stdtori}\big((0, \infty) \times S^1 \times S^1,   \ker(dz + x\ dy)\big).\end{equation}
Here $x\in (0, \infty)$ and $y,z\in S^1$. We note for later use that this contact manifold is a quotient of the subset $\{x>0\}$ in $(\mathbb{R}^3, \xi_\text{std})$, and the contactomorphism maps $z$ in $\mathbb{R}^3$ to $t$ in $M$.

We suppose throughout that any Legendrian knot $\Lambda$ in $(B, \pi, F, V)$ is disjoint from $B$ and transverse to the skeleton.  This is a generic condition and can easily be achieved by Legendrian isotopy. In this case, each point on $\Lambda \setminus \text{Skel}$ has a well-defined image on the Morse diagram under the flow by $\pm V$ to $\partial N(B)$; we call this image the \textit{front projection} and denote it  by $\mathcal{F}(V)$.  This front projection is smooth except at isolated cusps, and the $\frac{\partial t}{\partial x}$ slope lies in $(-\infty, 0]$.  (Here $x_i\in S_1$ parametrizes the binding component $B_i$.) At points where $\Lambda$ intersects the skeleton, $\mathcal{F}(\Lambda)$  teleports between paired trace curves.  The basic properties of the front projection are derived from the local contactomorphism of Theorem 1.1, and Theorem 1.5 in~\cite{GaLi15} establishes a complete set of front Reidemeister moves, which are shown in Figure~\ref{fig:reid}.  We remark that the move K3 shown here differs from ---and corrects--- the K3 move in the published version of \cite{GaLi15}; the change is  explained in \cite{GaLi19}.
\begin{figure}[htbp]
	\centering
	\includegraphics[width=\textwidth]{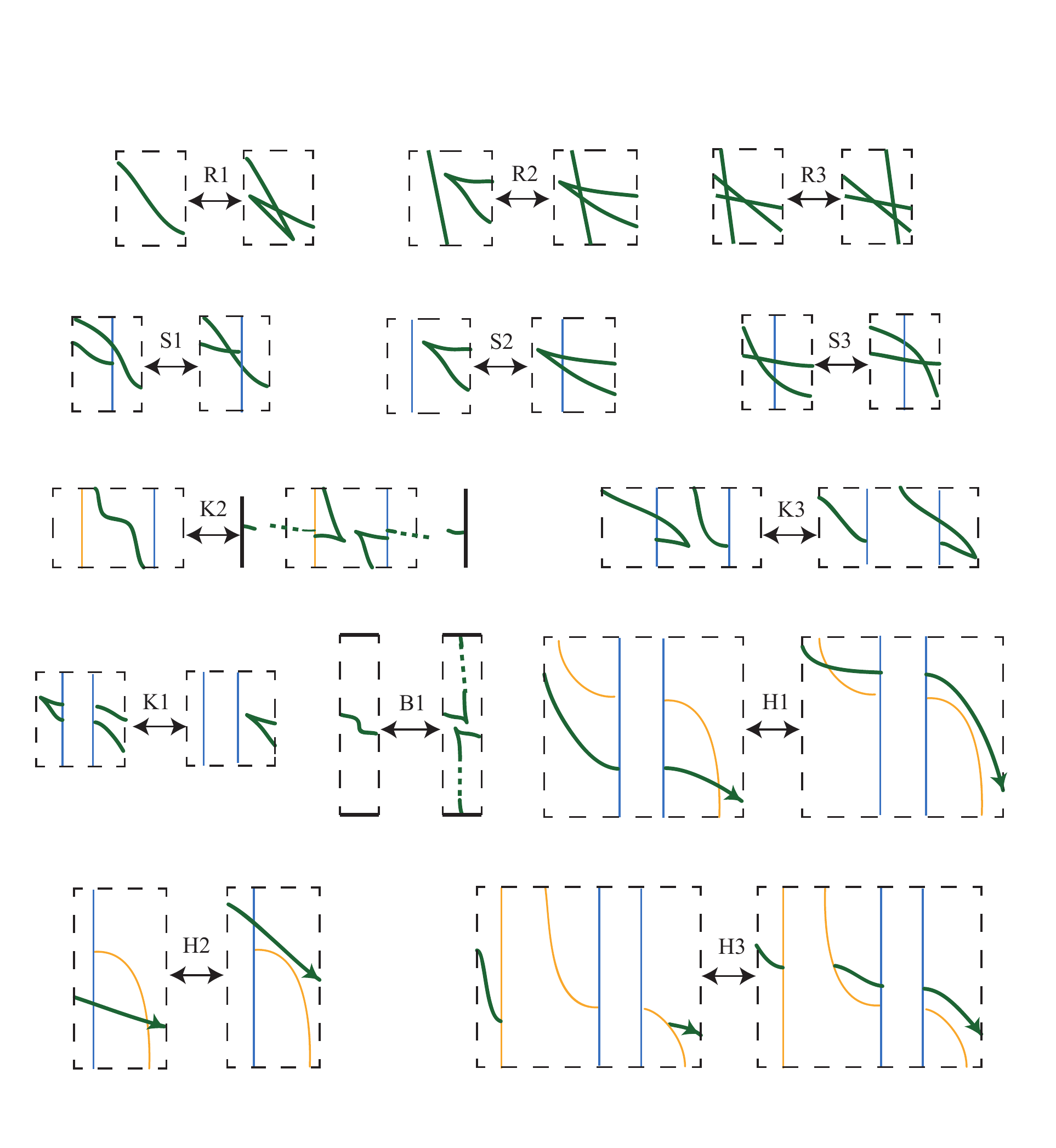}
	\caption{The Legendrian front Reidemeister moves}
	\label{fig:reid}
\end{figure}

Finally, for use in Section~\ref{sec:eucomp}, we recall the process for computing various first homology classes  from the data on the Morse diagram.  Assign paired trace curves on the Morse diagram opposite vertical orientations; this corresponds to orienting the dual core curves as generators of $H_1(\Sigma)$. Specifically, at an index one critical point, the pair (oriented core, co-core ray oriented towards $B$) should form a positive coordinate system with respect to the oriented page. Near $t=0$, label each trace curve on the Morse diagram with the corresponding homology generator.  When a trace curve labeled $A$ teleports across a trace curve labeled $B$ with the same orientation, their new labels above the teleporting points remain $A$ and become $A+B$, respectively.  If the orientations are opposite, the new labels are $A$ and $A-B$.  On $\Sigma_0$, select a flowline $\gamma_i$ from the index zero critical point $c_0$ to a marked point $p_i$ on each boundary component;  we identify $p_i$ with the left edge of the $i^{th}$ component of the Morse diagram.  Label the vertical edge by $P_i$ at $t=0$ and adjust the label by adding $\pm A$ when the trace curve labeled $A$ crosses it, with the same sign conventions as in the case of teleports. 

We have the following:

\begin{lemma}[Lemma 7.1, \cite{GaLi15}]\label{lem:homcomp}\hfill
\begin{enumerate}
\item $H_1(M) $ is generated by the trace curve labels;
\item For each $i$, $[\gamma_i-\phi_*(\gamma_i)] \in H_1(M)$ is computed by subtracting the bottom label on the left edge of the associated component from the top label of the edge.

\item The kernel of the map $H_1(\Sigma)\rightarrow H_1(M)$ is  generated by the set of differences  \[ 
[\gamma_i-\phi_*(\gamma_i)]-[\gamma_j-\phi_*(\gamma_j)]
\]
taken over all possible pairs $i, j$.    
 \item For $\Lambda\subset M\setminus \Sigma_0$, $[\Lambda]\in H_1(M)$ is the signed sum of labels of intersections between $\mathcal{F}(\Lambda)$ and the trace curves.
\end{enumerate}
\end{lemma}

Example~\ref{ex:exann} illustrates the labeling conventions and Example~\ref{ex:l0} includes a computation of $H_1(M)$ from a Morse diagram. 

\section{Computing the rotation and Euler numbers}

In this section we prove Theorems~\ref{thm:front} and \ref{prop:front}.  In each case, the result follows from constructing a Seifert surface for $\Lambda$ that lives in the complement of the vanishing locus of the vector field $V$.  Recall that $L_i$ denotes the link formed by index $i$ critical points of $V_t$ and  $\mathcal{L}$ was defined as $L_0 \cup - L_1$. In the case of a homology $3$-sphere, we compute the linking between   $\Lambda$ and $B\cup\mathcal{L}$, while the proof of Theorem~\ref{prop:front} follows from replacing linking numbers  with intersection numbers between curves and surfaces. 

\subsection{Homology $3$-spheres}

The quantity $\frac{1}{2}(D-U)+\lk(B \cup \mathcal{L},\Lambda)$ is preserved by all the Legendrian front Reidemeister moves, so it is immediate that this is an invariant of $\Lambda$ with respect to the given Morse structure on the open book.  However, without a combinatorial description of the space of all Morse structures supporting a given contact structure, it is not possible to prove directly that the quantity computed above is independent of the choices involved.

\begin{proof}[Proof of Theorem~\ref{thm:front}]
The  main idea in the proof follows Section~2.6.2 in~\cite{Et05}. Namely, in the complement of $B\cup\mathcal{L}$, the vector field $V$ is a non-vanishing section of the contact structure. We argue that when $\lk(B\cup\mathcal{L},\Lambda)=0$, $V$ defines a trivialization for $\xi$ over some  Seifert surface for $\Lambda\subset M\setminus (B\cup\mathcal{L})$.  In this case, the signed count of down and up cusps computes the rotation number just as in $\mathbb{R}^3$. Namely, cusps in the front projection correspond to tangencies between $T\Lambda$ and $V$, which in turn correspond to the intersections between the image of $T\Lambda$ and the $x$-axis in the trivialised bundle.  Just as in $(\mathbb{R}^3, \xi_\text{std})$, the vertical direction a cusp is traversed determines the sign of the corresponding intersection point, so the winding number is computed as half the difference between the number of down and up cusps.  

To apply this  argument, we must show that $\Lambda$ admits a Seifert surface disjoint from $\mathcal{L}$. 

This is accomplished by Lemmas \ref{lem:L1}, \ref{lem:L0},  and \ref{lem:intB} below, which collectively describe how to eliminate intersection points with the three types of components that constitute $B\cup\mathcal{L}$.   In each case,  the change in linking number is offset by the appropriate change in the signed cusp sum $D-U$, ensuring that the quantity computed by the formula  is preserved.  
\end{proof}

\begin{lemma}\label{lem:L1} Fix a component $L$ of $L_1$.  Then there is a Legendrian isotopy supported in  $M\setminus (B\cup\mathcal{L}\setminus L)$ that passes $\Lambda$ across $L$ and preserves the sum  $\frac{1}{2}[D-U]+\text{lk}( - L_1,\Lambda).$
\end{lemma}

\begin{proof}[Proof of Lemma~\ref{lem:L1}]

 Figure~\ref{fig:K3} below shows how the front projection changes as $\Lambda$ undergoes a Legendrian isotopy that culminates in passing $\Lambda$ across the relevant component of $L_1$. The first step creates new cusps via a Legendrian front Reidemeister I move and passes   one cusp across the given trace curve.  After a second Legendrian Reidemeister I move, a cusp   teleports across the same curve.  (Recall that this corresponds to $\Lambda$ intersecting the co-skeleton and skeleton, respectively, near the corresponding component of $L_1$.) Move S1 yields a front projection in the initial configuration for a K3 move, and the final step performs this K3 move.  This moves $\Lambda$ across $L_1$, changing the linking number, and one can easily verify that this also changes the count of up and down cusps by two. The desired change in linking is achieved by reflecting this sequence appropriately. 
\end{proof}
\begin{figure}[htbp]
	\centering
	\includegraphics[width=\textwidth]{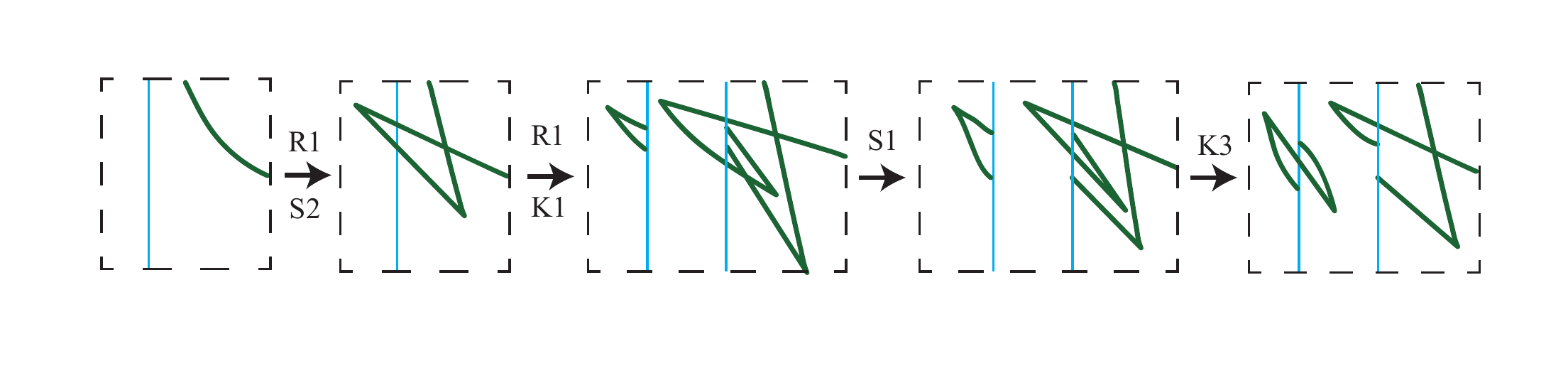}
	\caption{The sequence of Legendrian front Reidemeister moves labeling the arrows passes a segment of $\Lambda$ across the component of $L_1$ corresponding to the trace curves shown. }
	\label{fig:K3}
\end{figure}

\begin{lemma}\label{lem:L0} There is a  Legendrian isotopy of $\Lambda$ supported in $M\setminus (L_1\cup B)$  that reduces  $\text{lk}(L_0,\Lambda)$ by one and increases $D-U$ by $2$.
\end{lemma}

\begin{proof}
This is a direct consequence of Move K2, which corresponds to passing $\Lambda$ across $L_0$.
\end{proof}

\begin{lemma}\label{lem:intB}  

Fix a component $B_i$ of the binding.  Then there exists a Legendrian isotopy supported in $M\setminus (\mathcal{L} \cup B\setminus B_i)$ taking $\Lambda$ to $\Lambda'$  such that $\Lambda'$ admits a Seifert surface $H'$ with $H' \bullet B_i=0$ and

 \[ \frac{1}{2}\left(D-U\right)+\lk(B \cup \mathcal{L},\Lambda)=\frac{1}{2}\left(D'-U'\right)+\lk(B \cup \mathcal{L},\Lambda').\]
\end{lemma}

\begin{proof}

The existence of Move B1  implies that $\Lambda$ may be isotoped across any component of the binding.  This changes the linking number $\text{lk}(\Lambda, B)$ by one and the signed cusp sum by two.  Checking that these changes have opposite signs, we see that this move preserves the quantity $1/2(D-U)+\lk(B,\cup\mathcal{L},\Lambda)$.   This move is local: it may be applied to an arbitrarily small segment of $\mathcal{F}(\Lambda)$, and thus the associated isotopy is supported in the complement of $(B\setminus B_i)\cup \mathcal{L}$.

Given an orientable surface $H$ with $\partial H=\Lambda$, the isotopy may be extended to an isotopy of $H$ that removes or adds intersections between $H$ and $B_i$.  Any pair of intersections between $H$ and $B_i$ of opposite sign may be removed by adding a tube to $H$,  and operation which preserves the second homology class of the surface, so one may easily construct the desired $H'$ disjoint from $B_i$.
\end{proof}

\subsection{Arbitrary $3$-manifolds}\label{sec:aux}

As noted in the introduction, when the ambient manifold is not a homology $3$-sphere, the rotation number of $\Lambda$ is only defined relative to a fixed class in $H_2(M, \Lambda)$.  We next describe the construction of a link $X\subset M\setminus \Sigma_0$ which satisfies the property that $\Lambda\cup X$ is null-homologous in $\Sigma\times I$, as required for Theorem~\ref{prop:front}.

Suppose $\Lambda$ is a null-homologous Legendrian knot in $M$ and disjoint from the page $\Sigma_0$ of the open book.  Lemma~\ref{lem:Xdet} claims that the intersection $X$ between $\Sigma_0$ and any Seifert surface for $\Lambda$ suffices to determine the class of the Seifert surface in $H_2(M, \Lambda)$.  More precisely, for a fixed Seifert surface $H$ for $\Lambda$, let $X$ denote the link $H\cap \partial N(\Sigma_0)$.  By construction, $X\cup \Lambda$ is null-homologous in $M\setminus \Sigma_0\equiv \Sigma \times I$.  We now show that the homology class of $X \cup \Lambda$ recovers $[H]$.

\begin{proof}[Proof of Lemma~\ref{lem:Xdet}] Consider the long exact sequence of the pair $(\Sigma\times I,\Lambda \cup X)$:
\begin{equation*}
 \dots\rightarrow H_2(\Sigma\times I)\rightarrow H_2(\Sigma\times I,\Lambda \cup X)\stackrel{\delta}{\rightarrow} H_1(\Lambda \cup X)\rightarrow \dots  
 \end{equation*}

First observe that $H_2(\Sigma\times I)=0$, as $\Sigma$ has boundary, so the connecting map $\delta$ is injective.  There is at most one homology class of Seifert surfaces for $\Lambda \cup X$, but since $\Lambda \cup X$ is null-homologous, there exists at least one class.

Thus we get a unique homology class of Seifert surfaces for $\Lambda\cup X$ in $\Sigma\times I$.
\end{proof}

Given $\mathcal{F}(\Lambda)$, we note that it is possible to draw the front projection of an appropriate link $X$ on the Morse diagram without having to construct a Seifert surface and cut.  See Section 7.1 of \cite{GaLi15}.

With Lemma~\ref{lem:Xdet} in hand, the proof of Theorem~\ref{prop:front} follows from replacing linking numbers between curves with intersection numbers between curves and surfaces.

\subsection{The Euler class}\label{sec:eucomp}
In this section we prove Theorem~\ref{thm:eu}, which identifies the class of $\mathcal{L}$ as the Poincar\'e dual to the Euler class  of $\xi$.  

\begin{proof}[Proof of Theorem~\ref{thm:eu}]

Fix a class  in  $H_2(M)$ and represent it by a closed embedded surface $S$. After possibly isotoping $S$, it may be cut along  a Legendrian knot $\Lambda$ to get two Seifert surfaces $H$ and $H'$ of $\Lambda$ such that $S=H-H'$. According to Proposition~3.5.15 in~\cite{Ge08}, \[\langle \textrm{e}(\xi),H-H'\rangle=\rot_H(\Lambda)-\rot_{H'}(\Lambda).\]
We apply Theorem~\ref{prop:front} to compute the terms on the right, which yields the following:
			\begin{equation*}
			\langle \textrm{e}(\xi),S\rangle=\langle \textrm{e}(\xi),H-H'\rangle
		    =\rot_H(\Lambda)-\rot_{H'}(\Lambda)
			=\mathcal{L}\bullet (H-H')=\mathcal{L}\bullet S.
			\end{equation*}
			The claim follows.
				\end{proof}
		
			Together with Corollary 3.5.16 from~\cite{Ge08}, we see that $\mathcal{L}$ is null-homologous if and only if $\textrm{e}(\xi)=0$ if and only if $\rot(\Lambda)$ is independent of the choice of Seifert surface for all null-homologous Legendrian knots $\Lambda$ in $(M,\xi)$. We note  that although we defined $\mathcal{L}$ in terms of  a Morse structure, the homology class of this knot is rather simpler  to define.

\subsubsection{Computing the Euler class}
To conclude this section, we describe how to compute the homology class of $\mathcal{L}$ from the Morse diagram.  Since $\mathcal{L}$ is contained in the skeleton of $(B, \pi, F, V)$, it does not admit a front projection simply via the flow of  $\pm V$.  However, applying a carefully chosen isotopy class remedies this.  

\begin{lemma}\label{lem:l1comp}
For any component $L$ of $L_1$,  consider the vertical line connecting the top and bottom endpoints of a trace curve associated to the same index $1$ critical point.  Translate this horizontally by $\epsilon$ so that it is disjoint from the trace curve near $t=0$.  Then $[L]\in H_1(M)$ is the signed sum of intersections between this translated vertical line and the trace curves labeled with homology classes as in Lemma~\ref{lem:homcomp}.
\end{lemma}

\begin{proof}
 Recall that we may assume the index $1$ critical point $c$ associated to $L$ is fixed for all $t$.  
Push $L$ slightly off itself along one of the two flowlines to $B$. Since this push-off lies on the co-core  through $c$,  the front projection of $L$ is simply one of the trace curves associated to $c$.  In order to compute the homology class of the push-off, apply Move B1 to eliminate the intersection with the page $\Sigma_0$.  The result is isotopic to a copy of a single trace curve with a small neighborhood removed near $t=0$ and with the resulting endpoints connected by a vertical curve. The usual technique for computing homology may then be applied, and the only intersections will be with the vertical segment resulting from the B1 move.  For the purposes of computation, it thus suffices to consider this segment alone, as described in the statement of the lemma.
\end{proof}

Recall from Section~\ref{sec:back} that $\gamma_i$ is a flowline from $c_0$ to a marked point $p_i$ on $B$, and that the left edge of the $i^{th}$ component of the Morse diagram  is labeled by the relative homology class $[\gamma_i]=P_i\in H_1(\Sigma, \{p_i\}\cup c_0)$.  This vertical curve acts formally like a trace curve, in the sense that its label changes each time a trace curve ``teleports" from the left to the right edge or vice versa. 

\begin{lemma} For any choice of $i$, $[\gamma_i-\phi_*(\gamma_i)]=[L_0]$.  This class is computed by subtracting the bottom label of the left edge of the $i^{th}$ component of the Morse diagram from the top label.
\end{lemma}

The fact that $[L_0]$ may be computed for any choice of $i$ is unsurprising, as Lemma~\ref{lem:homcomp} implies that each difference $[\gamma_i-\phi_*(\gamma_i)]- [\gamma_j-\phi_*(\gamma_j)]$ is trivial in $H_1(M)$.

\begin{proof}  
 As in the case of $L_1$, we may apply a finger move along $\gamma_i$ and then across $B$ so that a knot isotopic to $L_0$ is contained in $M\setminus \Sigma_0$.  This knot consists of four distinct segments: the unchanged vertical segment in the skeleton, which is oriented with increasing $t$;  a vertical segment near $p_i\times I$ oriented downward; a copy of $\gamma_i$ at $t=1-\epsilon$; and a copy of $\phi_*(-\gamma_i)$ at $t=\epsilon$. Finally, isotope the original vertical segment along $\gamma_i\times I$, leaving a loop $\gamma_i-\phi_*(\gamma_i)$ near $t=\epsilon$. Applying Lemma~\ref{lem:homcomp}, the result follows.
\end{proof} 

\begin{example}[Computing the Euler class]\label{ex:exann}
Here we briefly illustrate the computational techniques described in the lemmas above.  The left image in Figure~\ref{fig:exann} shows the mapping cylinder for an overtwisted contact structure on $S^3$.  (Since $S^3$ is a homology sphere, the outcome is foregone, but the pictures are clearer than in the case of more complicated manifolds.)

\begin{figure}[h!]
	\centering
	\includegraphics[width=\textwidth]{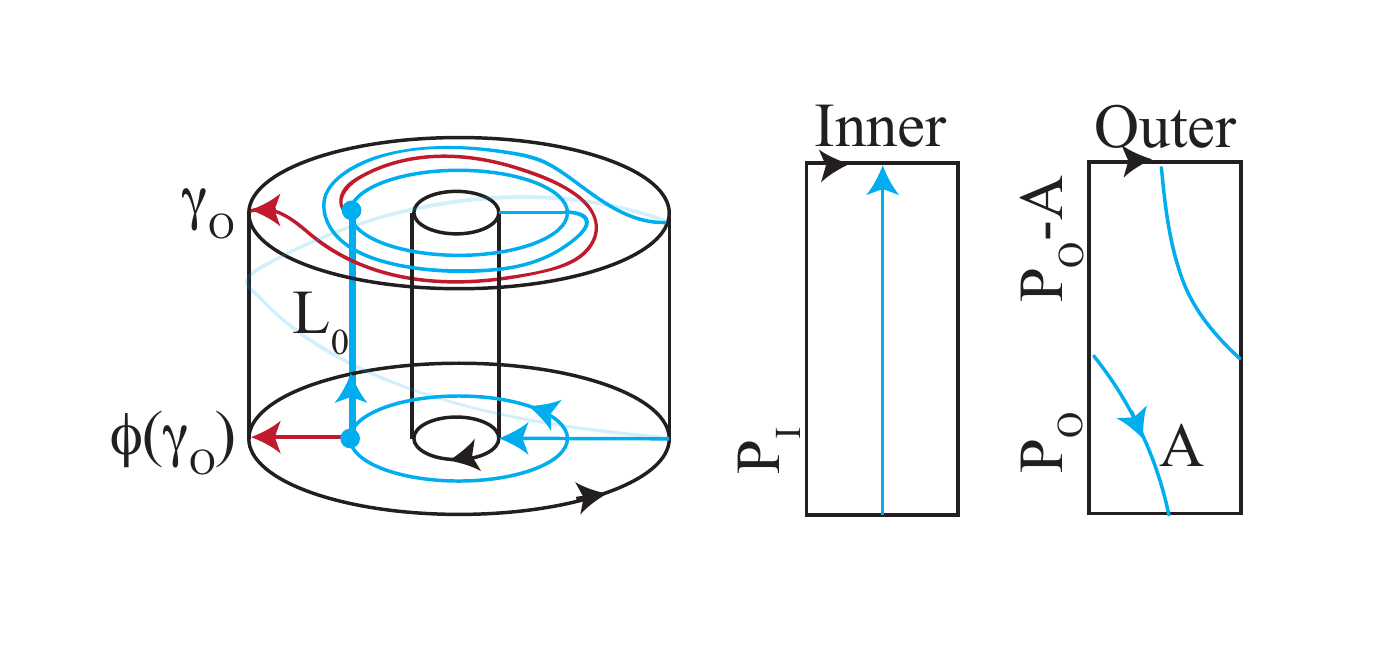}
	\caption{ The monodromy of the annular open book is a left-handed Dehn twist. }
	\label{fig:exann}
\end{figure}

Performing a finger move on $L_0$ across $B$ along $\gamma_0$ and then isotoping the resulting loop inside the mapping cylinder shows that $[L_0]=[\gamma_0-\phi(\gamma_0)]=-A$.  
\end{example}

When the monodromy of the open book is sufficiently simple, there are a few shortcuts for computing $[\mathcal{L}]$.  We say that a flowline  connecting $c_0$ to an index $1$ critical point $c$  is \textit{preserved} if it persists for all $t$.  This condition is equivalent to requiring no handleslides across the corresponding side of the co-core curve associated to $c$, and a preserved flowline traces out an annulus bounded by $L_0$ and the corresponding component of $L_1$.

\begin{lemma}\label{lem:cancel}
Suppose that at least one of the flowlines associated to a connected component $L\subset L_1$ is preserved. Then $[\mathcal{L}]=[\mathcal{L}\setminus (L_0\cup -L)]\in H_1(M)$. 
\end{lemma} 

Here, we treat the empty link as representing the $0$ class.

\begin{proof} A flowline is preserved if it is fixed by the monodromy map.  It thus defines an annulus in $M$ spanned by $L_0\cup-L$; since these cancel in homology, the class of $\mathcal{L}$ is unchanged by deleting both of them.  
\end{proof}

Note that preserved flowlines may be detected easily from the Morse diagram.
Any handle slide across the co-core appears on the Morse diagram as a point where one trace curve teleports across the trace curves associated to $c$.  If there are no teleporting points on at least one side of a fixed trace curve, then at least one of the associated flowlines is preserved. This may be seen in Example~\ref{ex:exann}.

This approach to computing the Euler class offers a simple proof of some familiar facts:

\begin{corollary}
Let $(B, \pi)$ be an open book for the contact manifold $(M,\xi)$. If the monodromy $\phi$ is the identity or $\Sigma$ is an annulus, then $e(\xi)=0$.
\end{corollary}

\begin{proof}
If $\phi$ is the identity, we may choose a Morse structure such that $V$ changes only by isotopy. It follows that each component of $\mathcal{L}$ is  isotopic to a meridian of the binding and is therefore nullhomologous.

 In the annular case, we may assume a Morse structure with a single index zero  and a single index one critical point on each page. Since there can be no handle slides, each flowline is preserved and $[L_0]=[L_1]$.
\end{proof}


\begin{example}\label{ex:l0}

In Figure~\ref{fig:euex}, the Morse diagram introduced in Figure~\ref{fig:firstex} is decorated with the labels required to compute the Euler class of the supported contact structure. $H_1(\Sigma)$ is generated by two elements $A$ and $B$, and we compute \[H_1(M)\cong \langle A \rangle \oplus \langle B \rangle  \slash \langle A \rangle \cong \langle B \rangle\cong \mathbb{Z}.\]

\begin{figure}[h!]
	\centering
	\includegraphics[width=\textwidth]{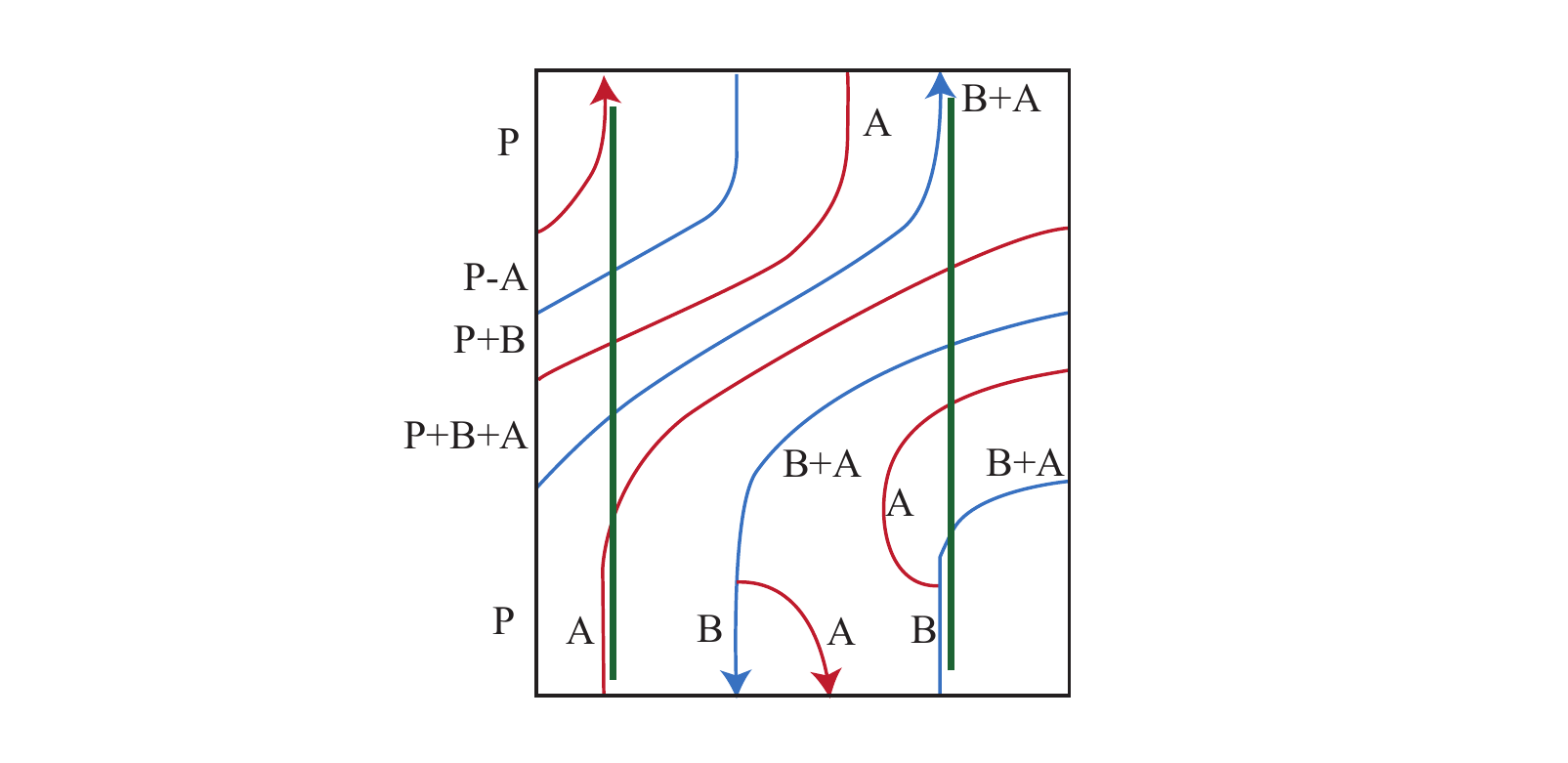}
	\caption{$e(\xi)=0$. }
	\label{fig:euex}
\end{figure}

The bold green curves are the translated vertical segments used to compute $[L_1]$.  Summing intersections with the left and right segments each yields  $0$. Subtracting the top and bottom labels on the left edge, we compute $[L_0]=0$.  Summing these with signs yields the result that $e(\xi)=0$.  Note that the computations for $\mathcal{L}$, $-L_1^1$, and $-L_1^2$ all agree, as required by Lemma~\ref{lem:cancel}.
\end{example}

\section{Linking and  intersection numbers}\label{sec:linking}

The previous section provides formulas for the rotation of $\Lambda$ in terms of intersection numbers between $B \cup \mathcal{L}$ and a Seifert surface $H$ (or $H_X$) for $\Lambda$.  This section presents algorithms to compute such intersections numbers, making Theorems~\ref{thm:front} and \ref{prop:front} effective computational tools. 

We begin with a few special cases where the computation of linking numbers is straightforward. 

\begin{proposition} Suppose that $\Lambda$ is a null-homologous Legendrian knot in $M$.  The  linking number between $\Lambda$ and $B$ is  computed by the signed intersection number between the front projection of $\Lambda$ and the horizontal (multi)curve $\{t=0\}$ on the Morse diagram:
\[ \text{lk}(\Lambda,  B)=\mathcal{F}(\Lambda)\bullet \{t=0\}.\]
\end{proposition}

\begin{proof} The page $\Sigma_0$ is a Seifert surface for $B$, so  intersection points between a generic Legendrian $\Lambda$ and $\Sigma_0$ are in bijection with intersection points between the images of these objects in the front projection.
\end{proof}

In order to compute intersection numbers more generally, we recall the generalization of Seifert's Algorithm introduced in \cite{GaLi15}. After perhaps introducing some auxiliary link $X$ as in Section~\ref{sec:aux}, we may assume that $\Lambda$ is null-homologous in $\Sigma \times I$.  Cutting $M$ along the skeleton of the Morse structure yields a collection of Legendrian arcs properly embedded in a set of solid tori.   Next, connect endpoints of these arcs with curve segments embedded in the skeleton. We call the resulting link $\widetilde{\Lambda}$.  The connecting segments are chosen to glue in pairs when the solid tori are reassembled into $M$; a Seifert surface $\widetilde{H}$ for $\widetilde{\Lambda}$ therefore induces a Seifert surface $H$ for $\Lambda$, and the intersection numbers of any curve with the two surfaces will agree.  

We now consider how this surgery affects the front projection.  Intersections between $\Lambda$ and the skeleton project to teleporting points on the Morse diagram.  Thus the operation just described connects pairs of teleporting endpoints of the front projection with segments  parallel to the trace curves.  

To make this more precise, observe that the teleporting endpoints of $\mathcal{F}(\Lambda)$ and the trivalent points of the trace curves separate each trace curve into intervals. Assign each interval $T_i$ a multiplicity $m_i(T_i)$ as follows.  

\begin{figure}[htbp]
	\centering
	\includegraphics[width=\textwidth]{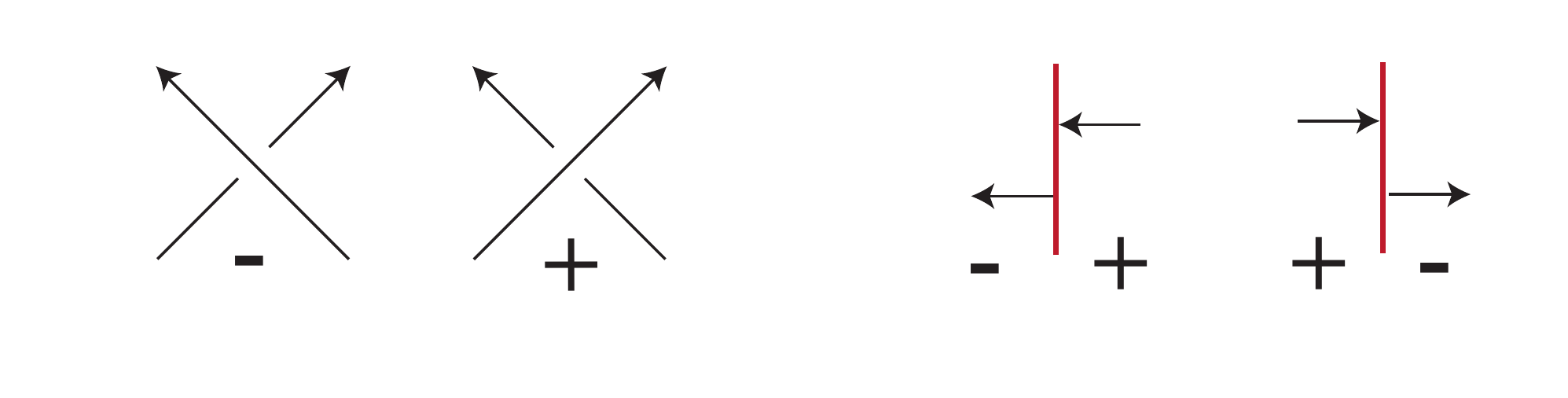}
	\caption{Left: signs for crossings of $\mathcal{F}(\Lambda)$.  Right: signs for teleporting endpoints of $\mathcal{F}(\Lambda)$ and trace curves.}
	\label{fig:signs}
\end{figure}

Assign the bottom interval of each trace curve multiplicity $0$ and assign paired trace curves opposite orientations. If there are no handle slides, define the multiplicity of any interval to be the sum of the multiplicity of the interval below and the sign of the teleporting endpoint separating them. When a handle slide occurs, add the multiplicities of the two converging branches to define the multiplicity of the interval above the trivalent point. Place $m_i(T_i)$ disjoint positively oriented segments parallel to and nearby $T_i$. These added segments may be extended to each other and to the teleporting endpoints of $\widetilde{F}(\Lambda)$ without introducing any new crossings into the diagram. See Section 7 of \cite{GaLi15}.

In order to actually construct $\widetilde{H}$, resolve crossings of $\mathcal{F}(\widetilde{\Lambda})$  as in the classical Seifert's Algorithm.  Call the resulting set of simple closed curves the \textit{total resolution} of $\Lambda$. These simple closed curves either bound discs in the Morse diagram or are isotopic to horizontal $t=c$ curves.  Note that  over- and undercrossings are unambiguous even when $\mathcal{F}(\Lambda)$ is drawn with double points. Legendrian curves whose front projections have  slopes closer to $0$ are farther from  the binding than Legendrian curves whose front projections have more negative slopes.  However, added curve segments parallel to the $T_i$ on the Morse diagram correspond to curves on the skeleton, so they cross under any front projection of a Legendrian segment.

\begin{lemma}\label{lem:intl0}  The signed intersection number  $L_0\bullet H=L_0 \bullet \widetilde{H}$ equals the signed sum of oriented horizontal curves in the total resolution of $\mathcal{F}(\Lambda)$ .  

\end{lemma}  

\begin{proof} The proof is an immediate consequence of the generalized Seifert's Algorithm just described.  When $M$ is reassembled by gluing the solid tori to the skeleton, each horizontal curve in the total resolution becomes a closed loop on a page that bounds a disc containing the index $0$ critical point. 
\end{proof}

\begin{lemma}\label{lem:intl1} Fix a component $L$ of $L_1$ and let $T$ be one of the two associated trace cuves on the Morse diagram.  Then $L\bullet H=\mathcal{F}(\Lambda)\bullet T$.
\end{lemma}

We defer the proof of this lemma until Section~\ref{sec:lagr}.

\begin{proposition} Let $\Lambda$ be a Legendrian link contained and null-homologous in $M\setminus \Sigma_0$ with Seifert surface $H$ and let $\Lambda'$ be a Legendrian link in $M$.  The  intersection number  $\Lambda'\bullet H$  equals the signed sum of crossings where $\mathcal{F}(\Lambda')$ passes over the total resolution of $\Lambda$.  
\end{proposition} 

\begin{figure}[htbp]
	\centering
	\includegraphics[width=\textwidth]{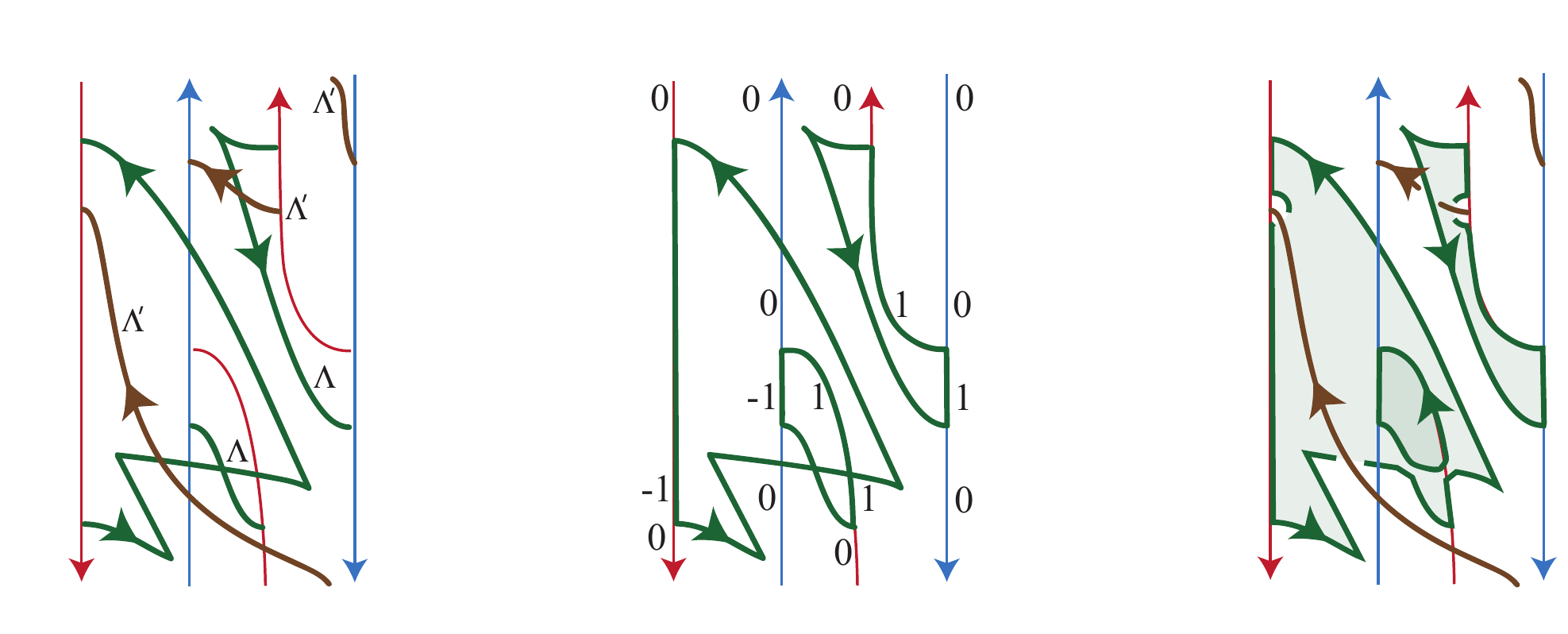}
	\caption{$\Lambda$ is null-homologous in $M\setminus \Sigma_0$, while $\Lambda'$ generates $H_1(M)$.  The middle picture shows the total resolution of $\Lambda$, with intervals of the trace curves labeled by their multiplicity. The right-hand picture shows that the signed intersection sum between $\Lambda'$ and the Seifert discs for $\Lambda$ is $-1$, and hence, $\Lambda'\bullet H=-1$.}	\label{fig:seifsurf}
\end{figure}

\begin{proof} Each closed curve in the total resolution of $\Lambda$ bounds a disc, and we claim that the signed sum of crossings counts intersections between $\Lambda'$ and these discs.  When both Legendrians lie in the interior of $M\setminus \text{Skel}$, this is immediate, so the only subtlety arises when $\Lambda'$ crosses the skeleton. To see that this has no effect on the linking, consider the segments added to $\Lambda$ which appear as parallel segments of trace curves in the total resolution.  These may be chosen to consist of two partial flowlines of $V$ connected via a vertical segment arbitrarily close to $L_0$.  Thus each added segment on the front corresponds to a curve of $\widetilde{\Lambda}$ passing behind $\Lambda'$. To assemble the discs into a Seifert surface $\widetilde{H}$ for $\widetilde{\Lambda}$, it suffices to glue them along twisted bands that recover the resolved crossings, and these can be assumed not to introduce any new intersections.  Finally, we construct $\widetilde{H}$ from $\widetilde{H'}$ by gluing along the curve segments added in the skeleton, observing that operations also preserves the count of intersection points.

If the original link $\Lambda$ was formed by adding auxiliary components to a link that was null-homologous in $M$ but not in $\Sigma\times I$, we may glue along the added components to recover a Seifert surface for the original link, again without changing the intersection number.  
\end{proof}

\subsection{Intersection with $L_0$ in regions of trivial monodromy}

The primary motivation for computing linking numbers is Theorem~\ref{thm:front}, and Lemma~\ref{lem:intl0} shows how to compute the intersection with $L_0$ even though it does not immediately admit a front projection.  Here, we note a further option that may  permit one to bypass constructing  the total resolution.  

Given a Morse structure on an open book $(B, \pi, F, V)$, the monodromy is \textit{locally trivial} for $t\in [t_0, t_1]$ if the trace curves on the Morse diagram are isotopic to vertical lines in the given $t$-interval.   Since handle slides occur only at isolated $t$-values, every Morse structure has many locally trivial regions.  

\begin{lemma}\label{lem:easylk} Suppose that $\Lambda$ is contained and null-homologous in $\Sigma\times [t_0, t_1]$ for some $t$-interval where the monodromy is locally trivial and let $H$ be a Seifert surface for $\Lambda$.  Then 
\[  L_0 \bullet H= \mathcal{F}(\Lambda) \bullet \{x_i=0\}.\] 

\end{lemma} 

\begin{proof}
When $\Lambda$ is contained in a region of trivial monodromy, the front projection and the Lagrangian projection are closely related.  Let $\pi: \Sigma\times [t_0, t_1]\rightarrow \Sigma_{t_0}$ denote the Lagrangian projection to $\Sigma_{t_0}$; one may always reparametrise to set $t_0=0$ if desired.   The condition that $\Lambda$ is null-homologous  in the cylinder translates to the condition that $\pi(\Lambda)$ is null-homologous in $\Sigma_{t_0}$.  Intersections between the Seifert surface $H$ and $L_0$ project to intersections between $\pi(H)$ and the index $0$ critical point $c_0$ on $\Sigma_{t_0}$, and these in turn correspond to intersections between $ \mathcal{F}(\Lambda)$ and a ray from $c_0$ to $B$.  Equivalently, we have $L_0 \bullet H= \mathcal{F}(\Lambda) \bullet \{x_i=0\}$, where $x_i$ is the parameter around some component of $B$ on the trace diagram.
\end{proof}

Although it may seem prohibitively restrictive to require the entire link to live in a $t$-interval of trivial monodromy, the next lemma demonstrates that this condition may always be realized via a Legendrian isotopy.  For maximum applicability, the lemma is stated here in a more general form than required for the purposes of this section.

\begin{lemma}[Standard Form Lemma]\label{lem:sfl}
Given any $\Lambda\subset (M, \xi)$ and $\epsilon >0$,   there exists  $\Lambda'\subset \Sigma\times [0, \epsilon]\subset M(\Sigma, \phi)$ where $M(\Sigma, \phi)$ is contactomorphic to $(M, \xi)$ and $\Lambda'$ is Legendrian isotopic to the image of $\Lambda$ under the given contactomorphism.
\end{lemma}

\begin{proof} We begin with the standard replacement of $(M, \xi)$ with a manifold constructed from a compatible open book with a Morse structure, and let $\Lambda''$ be the image of the original Legendrian under this map.  Perhaps after applying some Legendrian isotopy, we may assume this Legendrian is disjoint from  $\Sigma_0$; we continue to denote it by $\Lambda''$.  

Recall from Equation~(\ref{stdtori}) that each component of $M\setminus (\text{Skel}\cup B)$ is contactomorphic to $(0,\infty)\times S^1\times S^1$ with the contact form $dz+ x dy$.  Suppose first that $\Lambda''$ is contained in a single such component.  The condition that $\Lambda'' \cap \Sigma_0=\emptyset$ translates to the fact that its image in $W$ lies in $(0,\infty)\times S^1\times [z_0, z_1]$, and we may lift this to the cyclic cover $(0,\infty)\times \mathbb{R}\times  [z_0, z_1]\subset \mathbb{R}^3$.  We now apply a Legendrian isotopy to the resulting Legendrian $\widetilde{\Lambda}$ which pushes it into $(0,\infty)\times \mathbb{R}\times  [z_0, \frac{1}{2}z_1]$.   For example, this may be done simply by rescaling the $z$-coordinate of the front projection. Taking the quotient and including the interval into $S^1$ once again produces a Legendrian isotopy which halves the height of $\Lambda''$ in $M$.  

To show that such an isotopy exists for arbitrary $\Lambda''$ disjoint from $\Sigma_0$, it suffices to perform the same operation simultaneously on all components of $M\setminus (\text{Skel}\cup B)$ to preserve the continuity of $\Lambda''$ across intersections with the skeleton.  Iterating this process ensures that the original $\Lambda''$ is Legendrian isotopic to the resulting Legendrian link $\Lambda'$ that lies in an arbitrarily small $t$-interval of the open book.  
\end{proof}

\section{The generalized Lagrangian projection}\label{sec:lagr}

The classical Lagrangian projection in $(\mathbb{R}^3, \xi_{\text{std}})$ is the image of a Legendrian curve under projection to the $xy$-plane, and in the context of open books, projection to a page plays a similar role.  In this section, we project appropriate $\Lambda$ to a fixed page of the open book decomposition and examine the properties of this \textit{generalized Lagrangian projection}.

Consider the closure of $M\setminus \Sigma_0$.  The defining data of a Morse structure provide an explicit identification with $\Sigma\times I$, and collapsing the interval coordinate $t$ defines a projection from the complement of a single page to a copy of $\Sigma$.

\begin{defn} Given a Legendrian curve  disjoint from $\Sigma_0$ and parametrised as  $\Lambda(s)=\big(p(s),t(s)\big)\subset \Sigma\times I$,  the generalized Lagrangian projection is the image
\begin{equation*}
\Lambda_\Sigma(s):=p(s)\subset \Sigma_0.
\end{equation*}
\end{defn}

 Abusing notation, we will assume when necessary that over- and undercrossing data is recorded in the Lagrangian projection. 

As one would hope, this projection preserves contact geometric information, not simply topological data.  To see this, recall from Section~\ref{sec:back} that each component of $M\setminus (\text{Skel}\cup B)$ is contactomorphic to the quotient of an open subset of $(\mathbb{R}^3, dz+xdy)$ by the map identifying $z$ with $z+1$.  This contactomorphism $\Psi$ preserves the decomposition into pages, mapping $z$ in $\mathbb{R}^3$ to $t \in M$.  It follows that the generalized Lagrangian projection of $\Lambda$ on $\Sigma_0$ is the image of the standard Lagrangian projection of $\Psi^{-1}(\Lambda)$ on $\{z=0\}\subset \mathbb{R}^3$.  Thus many properties of the standard Lagrangian projection may be imported to $M$.  Although the contactomorphism is defined component-wise on $M\setminus (\text{Skel}\cup B)$,   the globally defined vector field $\partial_t$ extends it to any $\Lambda\subset M\setminus \Sigma_0$ that meets the skeleton transversely, which is a generic condition.

As in $\mathbb{R}^3$, the Lagrangian projection of a Legendrian link is harder to study than the front projection.  For example,  one can expect at best  a weak Reidemeister Theorem, since a generic isotopy of the Lagrangian in the page will not lift to a Legendrian isotopy.  However, the Lagrangian projection has the advantage that its curves are immersed.   Similarly,  one may recover geometric data about $\Lambda$ from $\Lambda_\Sigma(s)$ via integral computations as in the case of $(\mathbb{R}^3, \xi_\text{std})$; see, for example, \cite[Lemma~3.2.6]{Ge08}.  All  these claims are  easily derived using the toroidal charts described above. 	 

The Lagrangian projection provides a convenient tool for proving Lemma~\ref{lem:intl1}, which computes the intersection number between a Seifert surface and a fixed component $L$ of $L_1$: $L\bullet H=\mathcal{F}(\Lambda)\bullet T$.

\begin{proof}[Proof of Lemma~\ref{lem:intl1}]
 After possibly introducing an auxiliary link $X$, we may assume that the Seifert surface $H$ lies entirely in the cylinder $\Sigma \times I$, and thus admits a Lagrangian projection.  Then the intersection number $L\bullet H$ is simply the winding number of $\Lambda_\Sigma$ with respect to the associated index one critical point $c$.  This in turn is equal to the signed intersection number with a ray from $c$ to $B$, and without loss of generality, we may take this ray to be a flowline.  We now observe that as $t$ varies, the co-core flowlines may be fixed in a sufficiently small neighborhood of $c$; this is in contrast to the flowlines making up the core. Isotope $\Lambda_\Sigma$ via finger moves in $\Sigma_0$ along the co-core flowline until all intersection points lie in this  neighborhood of $c$.  Lift this isotopy to a (non-Legendrian) isotopy of $\Lambda$, noting that the count of intersection points between $\mathcal{F}(\Lambda)$ and the trace curve is preserved throughout.  It follows that the count of intersection points on the Morse diagram agrees with the count of intersection points in the Lagrangian projection, as desired.  
\end{proof}

At present, we are primarily interested in computing the classical invariants of an appropriate Legendrian knot from its  Lagrangian projection.  Recall  that Proposition~\ref{lem:classicalInvariantsInLagrangian} describes the Thurston-Bennequin and rotation numbers of $\Lambda$ in terms of the writhe and surface rotation of $\Lambda_\Sigma$, respectively.

\begin{proof}[Proof of Proposition~\ref{lem:classicalInvariantsInLagrangian}]

 Suppose that $\Lambda$ is a Legendrian knot contained and null-homologous in $\Sigma\times I$.  Let $H_\emptyset$ denote any representative of the unique homology class of a Seifert surface for $\Lambda$ in the cylinder $\Sigma\times I$.
 

The essential geometric fact is that any trivialization of $T\Sigma$ can be interpreted as a trivialization of $\xi$  over $\Sigma$.  Since $\Lambda_\Sigma$ is null-homologous in $\Sigma$, the classical Seifert's algorithm may be used to construct a Seifert surface $H_\emptyset$ from the Lagrangian projection so that the trivialization of $\xi$ lifts from $\Sigma$ to $H_\emptyset \subset \Sigma\times I$.  Since $\partial\Sigma\neq \emptyset$, such a trivialization  of $T\Sigma$ always exists.   

 To prove the formula for the Thurston-Bennequin invariant, observe that the contact framing and the page framing agree.  Thus, we compute the linking between $\Lambda$ and $\Lambda_t$, a transverse push-off in the $\partial_t$ direction,  As usual, this is given by the writhe of the original Lagrangian projection:
\[\tb(\Lambda):=\lk(\Lambda,\Lambda_t)=\writhe_\Sigma(\Lambda_\Sigma).\]

 By definition, $\text{rot}(\Lambda)$ is the rotation of $T\Lambda$ with respect to any trivialization of $\xi$ over $H_\emptyset$. The discussion above establishes that this is equivalent to the rotation of $T\Lambda$ with respect to some non-vanishing vector field and proves the first equality of the proposition.  However, we may tie this more closely to earlier sections of the paper by considering how this rotation  differs from the rotation computed with respect to the Morse structure vector field $V$. 

Consider the projection of some $H_\emptyset$ to $\Sigma_0$.  Generically, we expect that this image may include points where $V_0$ vanishes, so we identify an isotopy of vector fields that transforming  $V_0$ into some $V'$  which is non-vanishing on all points in the image of the projection of $H_\emptyset$.

\begin{figure}[htbp]
	\centering
	\includegraphics[width=\textwidth]{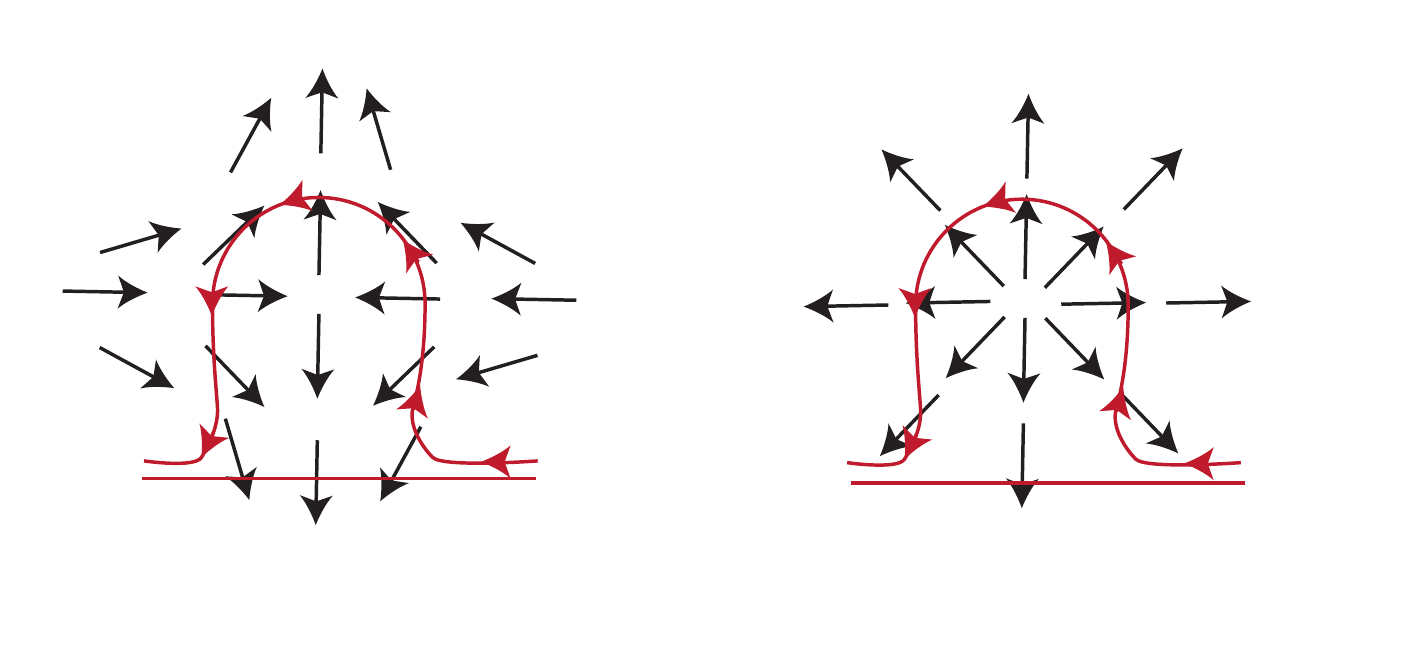}
	\caption{Index $1$ (left) and $0$ (right) critical points of $V$.  If the red curve bounds the projection of some Seifert surface $H_\emptyset$, we may directly compute the change in rotation number when $H_\emptyset$ is isotoped to lie in the complement of the vanishing locus of $V$.}
	\label{fig:vtox}
\end{figure}

By direct computation, we see that isotoping $c_0$ to lie outside the image of $H_\emptyset$ changes the rotation number by $+1$, while isotoping $V_0$ to move an index one critical point into  the complement of the image changes the rotation number by $-1$.  See Figure~\ref{fig:vtox}. Thus we may conclude that
\[\rot(\Lambda)=\rot_{Y}(\Lambda_\Sigma)= \rot_{V_0}(\Lambda_\Sigma)+ \mathcal{L}\bullet H_\emptyset.\]
\end{proof}

\begin{remark} 
Note that $\mathcal{L}\bullet H_\emptyset$ is simple to compute, as the winding number of $\Lambda_\Sigma$ around $c_i$ determines the multiplicity of the Lagrangian projection of $H_\emptyset$ there.
\end{remark}

\begin{remark}
	The formulas for the front and the Lagrangian projection are clearly compatible.   The Lagrangian projection is defined for   $\Lambda$ disjoint from $\Sigma_0$, so it is immediate that $\lk(B,\Lambda)=0$.  Moreover, points on $\Lambda$ which are mapped to down/up cusps under the front projection are mapped to positive/negative tangencies of $V_0$ under the Lagrangian projection. It follows that the term $1/2[D-U]$ in the front projection corresponds to $\rot_{V_0}(\Lambda_\Sigma)$.  Finally, observe that if $c_0$ and some $c_1$ lie in the same Seifert circle of a resolution of $\Lambda_\Sigma$, then standard cancellation moves may be applied to simultaneously remove both singularities; this should be compared to the fact that intersections of $H$ between $L_0$ and a component of $-L_1$ contribute canceling summands. 
\end{remark}


\begin{thebibliography} {111111111} 

\bibitem[DGS04]{DG2004} \textsc{F. Ding, H. Geiges, and A. Stipsicz}, Surgery diagrams for contact 3-manifolds, \textit{Turkish J. Math.} \textbf{28} (2004), 41--74.
    
\bibitem[DK18]{DuKe18} \textsc{S. Durst and M. Kegel}, Computing the rotation number in open books, \textit{J. G\"okova Geom. Topol. GGT} \textbf{12} (2018), 71--92.

\bibitem[DKK16]{DuKeKl16} \textsc{S. Durst, M. Kegel and M. Klukas}, Computing the Thurston-Bennequin invariant in open books, \textit{Acta Math. Hungar.} \textbf{150} (2016), 441--455.

\bibitem[Et05]{Et05} \textsc{J. Etnyre}, Legendrian and transversal knots, in: \textit{Handbook of Knot Theory} (W. Menasco and M. Thistlethwaite, eds.), Elsevier, Amsterdam (2005), 105--185.

\bibitem[EO08]{EtOz08} \textsc{J. Etnyre and B. Ozbagci}, Invariants of contact structures from open books, \textit{Trans. Amer. Math. Soc.} \textbf{360} (2008), 3133--3151.

\bibitem[GL15]{GaLi15} \textsc{D. Gay and J. Licata}, Morse structures on open books, \textit{Trans. Amer. Math. Soc.} \textbf{370} (2018), 3771--3802.

\bibitem[GL19]{GaLi19} \textsc{D. Gay and J. Licata}, Corrigendum to `Morse structures on open books', to appear in \textit{Trans. Amer. Math. Soc.}

 
\bibitem[Ge08]{Ge08} \textsc{H. Geiges}, \textit{An Introduction to Contact Topology}, Cambridge Stud. Adv. Math. \textbf{109}  (Cambridge University Press, 2008).

\bibitem[KMMM09]{KMMM09} \textsc{H. Kodama, Y. Mitsumatsu, S. Miyoshi, and A. Mori}, \textit{On Thurston's inequality for spinnable foliations}, in: Foliations, geometry, and topology, Contemp. Math.  \textbf{498} (Amer. Math. Soc., 2009), 173--193.

\bibitem[On18]{SO18} \textsc{S. Onaran}, Legendrian torus knots in lens spaces, \textit{Turkish J. Math.} \textbf{42} (2018), 936--948.

\end{thebibliography}
\end{document}